\documentclass[a4paper,preprint,english,noend]{article}
\usepackage[T1]{fontenc}
\usepackage[latin9]{inputenc}
\usepackage[affil-it]{authblk}
\usepackage{float}
\usepackage{color}
\usepackage{amsmath}
\usepackage{amssymb}
\usepackage{amsthm}              
\usepackage{algorithm}
\usepackage{algorithmic}
\usepackage{footnote}
\usepackage{diagbox}
\usepackage{hyperref}
\usepackage{tikz}
\usetikzlibrary{decorations,arrows} 
\usetikzlibrary{decorations.pathmorphing} 
\usepgflibrary{decorations.pathreplacing} 
\usetikzlibrary{mindmap,trees}
\usepackage{pgfpages}
\pgfpagesuselayout{resize to}[a4paper]
\usepackage{verbatim}
\usepackage{dcolumn}
\usepackage{enumitem}
\usepackage{url}

\newtheorem{thm}{Theorem}[section]

\newtheorem{defi}[thm]{Definition}

\newtheorem{rem}[thm]{Remark}
\newtheorem{convention}[thm]{Convention}
\newtheorem{notation}[thm]{Notation}

\newtheorem*{acknow}{Acknowledgments}

\definecolor{commentcolor}{rgb}{0.6, 0.6, 0.6}
\definecolor{myblued}{HTML}{34468c}
\definecolor{mybluel}{HTML}{697Fd4}
\definecolor{myredd}{HTML}{c31313}
\definecolor{myredl}{HTML}{c36767}
\definecolor{mygreend}{HTML}{2ca92c}
\definecolor{mygreenl}{HTML}{79a979}
\definecolor{myyellowd}{HTML}{e27722}
\definecolor{myyellowl}{HTML}{e2a97c}

\makeatother

\usepackage{babel}

\if@Mn@Math@
 \DeclareMathSymbol{\ell}{\mathlord}{letters}{'140}
\fi

\DeclareSymbolFont {mysymbols}{OMS}{cmsy}{m}{n} 
\DeclareMathSymbol{\calI}{\mathalpha}{mysymbols}{`I}
\DeclareSymbolFont {mylargesymbols}{OMX}{cmex}{m}{n}

\DeclareMathOperator{\lm}{lm}
\DeclareMathOperator{\lt}{lt}
\DeclareMathOperator{\lc}{lc}

\DeclareMathOperator{\sdeg}{s-deg}
\DeclareMathOperator{\sigdeg}{sig-deg}

\DeclareMathOperator{\lcm}{lcm}

\DeclareMathOperator{\idx}{index}

\DeclareMathOperator{\poly}{poly}

\DeclareMathOperator{\ap}{AP}
\DeclareMathOperator{\ggv}{G2V}

\DeclareMathOperator{\gvw}{GVW}
\DeclareMathOperator{\gvwhs}{GVWHS}

\DeclareMathOperator{\ff}{F5}
\DeclareMathOperator{\rs}{SB}

\DeclareMathOperator{\ffc}{F5C}

\DeclareMathOperator{\sba}{\textsc{Sba}}

\DeclareMathOperator{\fglm}{\textsc{FGLM}}

\DeclareMathOperator{\pot}{\prec_{pot}}
\DeclareMathOperator{\sch}{\prec_{s}}

\newcommand*{\openRight}{\csname .openrighttrue\endcsname}

\newcommand\rp{\ensuremath{\mathcal{R}}}

\newcommand\sigvar{e}
\newcommand\siglt{\ensuremath{\prec}}
\newcommand\sigle{\ensuremath{\preceq}}
\newcommand\siggt{\ensuremath{\succ}}

\newcommand\monset{\ensuremath{\mathcal{M}}}

\newcommand\lpset{\ensuremath{\mathcal{L}}}

\newcommand\spol[2]{\ensuremath{\mathcal{S}(#1,#2)}}
\newcommand\sig{\ensuremath{\mathrm{sig}}}

\newcommand\field{\ensuremath{\mathcal{K}}}

\newcommand\pring{\ensuremath{\mathcal{K}[x_1,\ldots,x_n]}}

\definecolor{commentcolor}{rgb}{0.6, 0.6, 0.6}
\definecolor{newcolor}{RGB}{218,22,20}

\definecolor{1}{RGB}{16,78,139}
\definecolor{2}{RGB}{0,114,49}
\definecolor{3}{RGB}{96,139,0}
\definecolor{4}{RGB}{255,100,0}
\definecolor{5}{RGB}{255,49,0}
\definecolor{6}{RGB}{125,0,0}

\begin{document}
\nocite{zobnin2010}
\nocite{panhuwang2012}
\nocite{singular314}
\nocite{gpSingularBook2007}
\nocite{bcr2011}
\nocite{blsr1999}
\nocite{fglmFGLM1993}
\nocite{wichmannFGLM1997}
\nocite{sw2010}
\nocite{sw2011}
\nocite{arsPhd2005}
\nocite{ahExtF52010}
\nocite{ckmConvertingBases1997}
\nocite{apF5CritRevised2011}
\nocite{galkinTermination2012}
\nocite{galkinSimple2012}
\nocite{gvwGVW2010}
\nocite{huang2010}
\nocite{bardetPhD}
\nocite{bardetComplexity2002}
\nocite{ACFP12}
\nocite{BFS04}
\nocite{BFS05}
\nocite{volnyGVW2011}
\nocite{fF419999}
\nocite{bardetPhD}
\nocite{apF452010}
\nocite{fF41999}


\title{An analysis of inhomogeneous signature-based Gr\"obner basis
  computations}

\date{\empty}

\author{Christian Eder}
\affil{\small{INRIA, Paris-Rocquencourt Center, PolSys Project\\
UPMC, Univ. Paris 06, LIP6\\
CNRS, UMR 7606, LIP6\\
UFR Ing\'enierie 919, LIP6\\
Case 169, 4, Place Jussieu, F-75252 Paris\\ \vspace*{2mm} Christian.Eder@inria.fr}}

\maketitle
\begin{abstract}
In this paper we give an insight into the behaviour of
signature-based Gr\"obner basis algorithms, like $\ff$, $\ggv$ or $\rs$, 
for inhomogeneous input. On the one hand, it seems that the restriction to
sig-safe reductions puts a penalty on the performance. 
The lost connection between polynomial degree and signature degree
can disallow lots of reductions and can lead to an overhead in the computations.
On the other hand, the way critical pairs are sorted and corresponding
s-polynomials are handled in signature-based algorithms is a very efficient one,
strongly connected to sorting w.r.t. the
well-known sugar degree of polynomials.

\end{abstract}

\section{Introduction}

Gr\"obner bases are a fundamental tool in computer
algebra. In 1965 Buchberger introduced a first algorithmic attempt for their computation, see~\cite{bGroebner1965}. 

In \cite{fF52002Corrected} Faug\`ere introduced the $\ff$ Algorithm which uses the concept of {\it signatures}
to detect zero reductions efficiently during the computation of Gr\"obner bases. 
In the last couple of years, several variants and
optimizations in the class of signature-based algorithms,
for example, $\ffc$ (\cite{epF5C2009}), $\ggv$ (\cite{ggvGGV2010}) or $\rs$
(\cite{rounestillman2012}) have been developed. 
Whereas the above mentioned publications focus mainly on the area of optimizing
signature-based criteria for detecting useless critical pairs, a close look at
the overall behaviour of signature-based computations in general is still
missing. Here we want to fill this gap and discuss advantages and disadvantages
of the signature-based attempt. Without going into detail about efficient
implementations we analyze the underlying characteristics all signature-based
algorithms share:
\begin{enumerate}
\item Sorting critical pairs by increasing signatures, and
\item processing only so-called sig-safe reduction steps. 
\end{enumerate}

By doing this we clear up myth that signature-based algorithms are only
applicable for homogeneous input data, but that they are not useful (either in the sense of
being incorrect or in the sense of being under-performing) in the
inhomogeneous setting. 

In Section~\ref{sec:Background} we introduce the basic setting for
signature-based Gr\"obner basis algorithms. There we unify the fundamental
framework for such algorithms, describing the differences to a
pure polynomial approach. 
Making some smaller changes to the initial
presentation of $\ff$ in Section~\ref{sec:ff}, Gr\"obner bases for inhomogeneous
input can be computed, too. Even more, it turns out that with
this new description understanding the algorithm's inner workings is much
easier. Following this, we give for the first time a discussion about the 
strong connections between the sorting of critical pairs by increasing signatures 
and the corresponding sorting by the so-called sugar
degree. The sugar degree, introduced in \cite{gmnrtSugarCube1991} and 
further discussed in \cite{bcr2011}, is known to be a powerful tool optimizing 
pure polynomial Gr\"obner basis computations in the inhomogeneous setting. Thus, 
explaining its relation to the signature-based world, we are able to allow a 
first estimate for the usefulness of those kind of algorithms beyond the homogeneous 
case. Even more, in Section~\ref{sec:problems} we test the differences in the behaviour 
of sig-safe reductions comparing $4$ different implementations of signature-based 
algorithms for a wide range of examples side-by-side in the respective inhomogeneous
and homogenized version. It turns out that, when compared to pure polynomial
attempts, there is no built-in disadvantage relying on signatures for computing 
Gr\"obner bases of inhomogeneous input.

The main contribution of this paper is to give a deeper insight into the inner
workings of signature-based Gr\"obner basis algorithms with a view towards 
optimizing the order in which critical pairs are handled. Besides this, we give 
first ideas for good heuristics to decide when to use which variant in order
to benefit from a better performance.

\section{Basic setting}
\label{sec:Background}
Let $n\in\mathbb{N}$, $\field$ a field, and
$\rp=\pring$. 
Furthermore, we denote the monoid of all monomials in $x_1,\dotsc,x_n$ by
$\monset := \left\{ \prod_{i=1}^n x_i^{\alpha_i} \mid (\alpha_1,\dotsc,\alpha_n)
\in \mathbb{N}^n\right\}$. We mostly use the shorthand notation $x^\alpha :=
\prod_{i=1}^n x_i^{\alpha_i}$. A polynomial $p \in \rp$ is a finite $\field$-linear
combination of monomials in $\rp$,
$p = \sum_{\alpha \in U} c_\alpha x^\alpha,
  c_\alpha \in \field$, $U$ a finite subset of $\mathbb{N}^n$.
We define the {\bf degree of a polynomial}
$p \neq 0$\footnote{For the zero polynomial we set $\deg(0) = - 1$.} by 
$\deg(p) = \max\left\{ \sum_{i=1}^n \alpha_i \mid c_{\alpha} \neq
0\right\}$.
We say that a polynomial $p$ is {\bf homogeneous}, if all its monomials have the
same degree; otherwise we call $p$ {\bf inhomogeneous}.

Let $F=\left(f_{1},\ldots,f_{m}\right)$,
where each $f_{i}\in \rp$, and $I=\langle F\rangle \subset \rp$ is 
the ideal generated by the elements of $F$. 

Moreover, fixing a well-ordering $<$ on $\monset$ we get a unique 
representation of the elements in $\rp$:
For a polynomial $p\in \rp$, we denote 
$p$'s {\bf leading monomial} by $\lm\left(p\right)$,
its {\bf leading coefficient} by $\lc\left(p\right)$, and 
write $\lt\left(p\right)=\lc\left(p\right)\lm\left(p\right)$ for its {\bf leading
term}.
In particular, a well-ordering $<$ preferring the degree over any other 
criterion to sort elements is denoted {\bf degree compatible ordering}.

Let $\sigvar_{1},\ldots,\sigvar_{m}$ be the canonical generators
of the free $\rp$-module $\rp^{m}$. We define a map
\begin{center}
$
\begin{array}{rccc}
\nu : &\rp^{m} &\rightarrow &I\\
    &\sum_{i=1}^{m} p_i\sigvar_i &\mapsto &\sum_{i=1}^m p_if_i,
\end{array}
$
\end{center}
$p_i \in \rp$ for all $1 \leq i \leq m$. 
Thus we extend the ordering $<$ to an admissible ordering $\siglt$ on the set
$\monset' := \left\{ t \sigvar_i \mid~t\in \monset, 1\leq i \leq m\right\}$. Without any restriction, 
the reader can think of the following two choices for $\siglt$ in the following:
\begin{enumerate}
\item preferring the module position over the term  $\pot$: \\
$t_i\sigvar_{i}\siglt t_j\sigvar_{j}$
iff $i<j$, or $i=j$ and $t_i < t_j$.
\item Being induced by $<$, the Schreyer ordering $\sch$: \\
$t_i\sigvar_{i}\siglt t_j\sigvar_{j}$
iff $t_i \lm\left(\nu(\sigvar_{i})\right)<t_j \lm\left(\nu(\sigvar_{j})\right)$,
or $t_i \lm\left(\nu(\sigvar_{i})\right) = t_j \lm\left(\nu(\sigvar_{j})\right)$ and $i < j$.
\end{enumerate}
In \cite{gvwGVW2010} it is shown that the above orderings are the most efficient 
ones for sig\-na\-ture-\-based Gr\"obner basis computations. The author has 
made similar experiences in various tests of his implementations, see
Section~\ref{sec:problems} for more details.

Most of the considerations in this paper are independent of the chosen
extended ordering, thus we use the notation $\siglt$ and specify to $\pot$
respectively $\sch$ whenever differences appear. The notions of leading
monomial, leading term, and leading coefficient generalize naturally to $\rp^m$
w.r.t. $\siglt$ on $\monset'$. Additionally, for $0 \in \rp^m$ we define $\lm(0) = \lt(0) =
0$.


\begin{notation}
For an easier description in the following let us agree on
the notation $\lpset := \monset' \times I$.
\end{notation}

\begin{defi}
\label{def:lpoly}
Let $p$ be a polynomial in $I$.
\begin{enumerate}
\item Let $h = \sum_{i=1}^m h_i e_i \in \rp^m$ be such that
$p = \nu(h)$. 
We say that $\lm(h) \in \monset'$ is {\bf a signature} of $p$.
Moreover, considering a well-ordering 
$\siglt$ on $\monset'$ there exists for each $p\in\rp$ a unique, minimal
signature.
\item An element $f = (t\sigvar_i,p) \in \lpset$ is called a {\bf labeled
  polynomial}, if $t \sigvar_i$ is a signature of $p$. For a labeled polynomial $f=(t\sigvar_i,p)$ we define the shorthand
  notations $\poly(f) = p$, $\sig(f) = t\sigvar_i$, and $\idx(f) = i$. Talking about the leading
  monomial, leading term, leading coefficient, degree, and least common
  multiples of $f \in \lpset$ we
  always assume the corresponding value of $\poly(f)$. Furthermore, if $G =
  \{g_1,\dotsc,g_{\ell}\} \subset \lpset$, then we define
  $\poly(G) := \left\{\poly(g_1),\dotsc,\poly(g_{\ell})\right\} \subset I$.
\item Let $f \in \lpset$, let $t \in \monset$, and let $c \in \field$. We define a
multiplication of $f$ by $ct$ via
$ctf := \left(t \sig(f),ct \poly(f)\right) \in \lpset$.
\item A {\bf critical pair} of two labeled polynomials $f$ and $g$ is a tuple
$(f,g) \in \lpset^2$. $\deg(f,g) := \deg\left(\lcm\left(\lm(f),\lm(g)\right)\right)$ 
defines the {\bf degree of a critical pair}.
Moreover, we define the
{\bf s-polynomial} of two labeled polynomials $f$ and $g$ in $\lpset$ by 
\[\spol{f}{g}=\left(\omega,\lc(g)u_f \poly(f)-\lc(f)u_g
\poly(g)\right)\]
where $\omega = \lm\left(
u_f\sig(f) - u_g\sig(g) \right)$ and $u_h
=\frac{\lcm\left(\lm(f),\lm(g)\right)}{\lm\left(h\right)}$ for $h \in \{f,g\}$. 
$\spol f g$ is called {\bf non-minimal} if $\omega \siglt 
\max\left\{u_f\sig(f),u_g\sig(g) \right\}$.
\item We define the {\bf signature degree} of a labeled polynomial $f = (te_i,p)\in \lpset$
by
\[ \sigdeg(f) := \deg(t) + \deg\left(\nu(e_i)\right).\]
Moreover, in the following it makes sense to speak about the signature degree
of a critical pair: $\sigdeg(f,g) := \sigdeg\left( \spol f g \right)$.
\end{enumerate}
\end{defi}

Next we extend the notions of reduction and standard representation from the
pure polynomial setting to the signature-based one:

\begin{defi}
\label{def:standard}
Let $f,g \in \lpset$ be two labeled polynomials. Moreover, let $G\subset \lpset$.
\begin{enumerate}
\item We say that {\bf $f$ reduces sig-safe to $g$ modulo $G$} if there exist
$r_0=f,\dotsc,r_k=g \in \lpset$ such that for all $i \in \{1,\dotsc,k\}$
there exist $g_{j_i} \in G$, $t_{i} \in\monset$ and $c_{i} \in\field$
fulfilling
\begin{enumerate}
\item $r_{i}=r_{i-1}-c_{i}t_{i}g_{j_{i}}$,
\item $\lm\left(r_{i}\right)<\lm\left(r_{i-1}\right)$, and
\item $t_i\sig\left(g_{j_{i}}\right)\siglt\sig\left(r_{i-1}\right)$.
\end{enumerate}
\item $f$ has a {\bf standard representation with respect to $G$} if there
exist $h_{1},\ldots,h_{k}\in \rp$, $g_1,\dotsc,g_{k} \in G$  such that
$\poly(f)=\sum_{i=1}^k h_{i}\poly(g_{i})$, and for each $i \in \{1,\ldots,k\}$
either $h_{i}=0$, or
\begin{enumerate}
\item $\lm\left(h_{i}\right)\lm\left(g_{i}\right)\leq\lm\left(f\right)$, and 
\item $\lm\left(h_{i}\right)\sig(g_{i})\sigle \sig(f)$. 
\end{enumerate}
\item If there exists $h\in G$ such that $\sig(h) \mid \sig(f)$ and $\lm(h) \mid 
\lm(f)$, then we say that {\bf $f$ is sig-redundant to $G$}.
\end{enumerate}
\end{defi}

Clearly, if $f$ reduces sig-safe to $0$ modulo $G$, then it has a standard representation
w.r.t. $G$.

The restriction of the reducer $g_{j_{i}}$ by 
$t_i\sig\left(g_{j_{i}}\right)\siglt\sig\left(r_{i-1}\right)$ in
each step of a sig-safe reduction is essential for the correctness of signature-based algorithms.  
If a labeled polynomial $f$ has a standard representation w.r.t. $G$, then $\poly(f)$ has a 
standard representation w.r.t. $\poly(G)$\footnote{Due to signature restrictions the inverse does not
necessarily hold.}.
Thus we can give a statement similar 
to Buchberger's Criterion, see~\cite{bGroebner1965}, for the signature-based
setting.

\begin{thm}
\label{thm:gb} 
Let $G = \left\{g_1,\dotsc,g_{k}\right\}
\subset \lpset$ such that $\left\{f_1,\dotsc,f_{m}\right\}
\subset \poly(G)$. If for each pair $(i,j)$ with $i>j$, $1\leq i,j
\leq k$, either
\begin{enumerate}
\item $\spol {g_i}{g_j}$ is non-minimal, or
\item $\spol{g_i}{g_j}$ has a standard representation w.r.t. $G$,
\end{enumerate}
  then $\poly(G)$ is a Gr\"obner basis of $I$.
\end{thm}

\begin{proof}
See, for example, ~\cite{epF5C2009,epSig2011}.
\end{proof}

\begin{rem}
It is well-konwn that
non-minimal elements are useless for the resulting Gr\"obner basis as well as
for the intermediate computations in signature-based algorithms. We refer to 
\cite{epSig2011} for more details on this fact.
\end{rem}

Next we present a generic signature-based Gr\"obner basis algorithm lying an 
emphasis on the general ideas behind signature-based computations.
Proofs of correctness and termination of Algorithm~\ref{alg:sba} can be
found in~\cite{epSig2011}, Theorem~14.

\begin{algorithm}[h]
\caption{\label{alg:sba} Generic signature-based Gr\"obner basis algorithm
w.r.t.  $<$ ($\sba$)}
\begin{algorithmic}[1]

\REQUIRE{$F=(f_1,\dotsc,f_m)$ a finite sequence of elements in $\rp$}
\ENSURE{$\poly(G)$ a Gr\"obner basis for $\langle F \rangle$ w.r.t. $<$}
\STATE{$G \gets \emptyset$, $P\gets \emptyset$}
\FOR{$(i=1,\dotsc,m)$}
\STATE{$g_i \gets (\sigvar_i,f_i)$}
\STATE{$G \gets G \cup \{g_i\}$}
\STATE{$P \gets P \cup \left\{ ( {g_i}, {g_j}) \mid g_i,g_j \in G,
  j<i\right\}$}\label{alg:pairs}
\ENDFOR
\WHILE{$(P \neq \emptyset)$}
\STATE{Let $(f,g )\in P$ such that $\spol f g$ has minimal signature w.r.t.
  $\siglt$.}\label{alg:minchoice}
\STATE{$P \gets P \backslash \left\{(f, g)\right\}$}\label{alg:minchoice2}
\STATE{Reduce $\spol f g$ sig-safe to $r$.}
\IF{$(\poly(r) \neq 0 \text{ and } r \text{ is not sig-redundant to }G)$}
\STATE{$P \gets P \cup \left\{(r, h) \mid h \in G, \spol r h \textrm{ not
  non-minimal}\right\}$}
\STATE{$ G \gets G \cup \{ r \}$}
\ENDIF
\ENDWHILE
\RETURN{$\poly(G)$}
\end{algorithmic}
\end{algorithm}

As in the pure polynomial setting a Gr\"obner basis algorithm without
any criteria to detect not necessary computations in advance, like Algorithm~\ref{alg:sba}
represents, is not efficient. In the signature-based world there exist
two main criteria to detect useless critical pairs:
\begin{enumerate}
\item the {\bf non-minimal signature criterion}, based on already
known syzygies: It checks if the leading monomial of a syzygy
divides the signatures of a critical pair;
\item the {\bf rewritable signature criterion}, based on the fact that for any signature
only one polynomial needs to be computed. 
\end{enumerate}

The most efficient implementations of signature-based Gr\"obner basis 
algorithms nowadays are

\begin{enumerate}
\item Faug\`ere's $\ff$ Algorithm (\cite{fF52002Corrected}) and optimizations
(\cite{epF5C2009,egpF52011,epSig2011}), 
\item Gao, Guan and Volny's $\ggv$ Algorithm (\cite{ggvGGV2010}),
\item Gao, Volny and Wang's $\gvw$ Algorithm (\cite{gvwGVW2010,volnyGVW2011}),
  and
\item Arri and Perry's algorithm (\cite{apF5CritRevised2011}) respectively Roune and
Stillman's optimized version $\rs$
(\cite{rounestillman2012}).
\end{enumerate}

The first two mainly differ in their usage and implementation of the above 
mentioned signature-based criteria to detect 
useless critical pairs. They share $\pot$ as ordering used on the 
signatures which leads to an incremental (w.r.t. the input sequence $F$) computation of Gr\"obner bases.

However $\rs$ as well as $\gvw$ are capable of using 
different orderings on the signatures. In
\cite{gvwGVW2010} the authors show that the Schreyer ordering $\sch$ turns out
to be the most efficient one for a wide range of example classes. Due to the
fact that $\sch$ does not favour the position of the module element over the
corresponding term a non-incremental computation is achieved.

For the focus of this paper we are neither interested in the specific
variants these criteria can be implemented nor in a comparison of those in terms
of efficiency or timings. It is enough to keep in mind that both criteria are
based on the signatures of labeled polynomials considered
during the algorithm's workings. Here we focus on the connection
of purely polynomial data to the signatures.

\begin{rem}
If we assume $\siglt\; = \pot$, then Algorithm~\ref{alg:sba} computes a
Gr\"obner basis of $\langle F \rangle$ incrementally, storing the
critical pairs of higher index in $P$, but prolonging their reduction until
all elements of lower index have been processed.

If there exist several critical pairs in $P$ of the same signature in
Line~\ref{alg:minchoice}, choose the one that entered $P$ first.
\end{rem}

\begin{convention}
In the following we often speak about s-polynomials in $P$ meaning the s-polynomial 
of a corresponding critical pair in $P$. Moreover, for any Gr\"obner basis algorithm we assume $F$ as input.
\end{convention}

Investigating the algorithms' behaviour for inhomogeneous input
data we can focus mainly on the handling of a single s-polynomial: 
Generate an s-polynomial and compute a sig-safe reduction step of that s-polynomial.
Thus we need not specify $\siglt$ in the following. 

\section{Problems of inhomogeneous signature-based computations depending on
  pure polynomial data}
\label{sec:ff}
The $\ff$ Algorithm as presented in \cite{fF52002Corrected} is restricted to homogeneous
input data. None of its successors, like $\ggv$ or $\rs$ have this
restriction. So what is the decisive factor here? Signature-based Gr\"obner basis 
computations, in particular the efficiency of the signature-based criteria rely
on the fact that s-polynomials are handled by increasing signature.

$\ff$, as presented in \cite{fF52002Corrected} chooses s-polynomials differently
from Algorithm~\ref{alg:sba}: Instead of picking the next s-polynomial from $P$ 
w.r.t. minimal signature (Line~\ref{alg:minchoice}), 
$\ff$ uses a presorting of $P$ by the degree of the corresponding s-po\-ly\-no\-mials.
To mimic this one needs to change Algorithm~\ref{alg:sba} beginning in
Line~\ref{alg:minchoice}:

\begin{algorithm}[h]
\caption{\label{alg:ff1} Presorting changes for $\ff$}
\begin{algorithmic}[1]
\STATE{$\dotsc$}
\STATE{$d \gets \min\left\{\deg(f, g) \mid (f, g) \in P\right\}$}
\STATE{$Q \gets \left\{ (f, g) \mid \deg(f, g ) =
  d\right\}$}\label{alg:presort}
\STATE{$P \gets P \backslash Q$}
\WHILE{$(Q \neq \emptyset)$}
\STATE{Let $(f, g) \in Q$ such that $\spol f g$ has a minimal signature w.r.t.
  $\siglt$.}\label{alg:minchoicef5}
\STATE{$Q \gets Q \backslash \left\{(f, g)\right\}$}
\STATE{$\dotsc$}
\ENDWHILE
\end{algorithmic}
\end{algorithm}

Changing Algorithm~\ref{alg:sba} as explained above is not enough 
to ensure the correctness of the resulting algorithm. 
Replacing the corresponding parts of Algorithm~\ref{alg:sba} with the pseudo code 
of Algorithm~\ref{alg:ff1} we need to distinguish where
newly generated critical pairs are stored. This postsorting is explained in
Algorithm~\ref{alg:ff2}.

\begin{algorithm}[h]
\caption{\label{alg:ff2} Postsorting changes for $\ff$}
\begin{algorithmic}[1]
\STATE{$\dotsc$}
\IF{$(\poly(r) \neq 0 \text{ and } r \text{ is not sig-redundant to }G)$}
\STATE{$P \gets P \cup \left\{(r, h) \mid h \in G, \spol r h \textrm{ not
  non-minimal}, \deg(r,h )>d\right\}$}
\STATE{$Q \gets Q \cup \left\{(r, h) \mid h \in G, \spol r h \textrm{ not
  non-minimal}, \deg(r, h )=d\right\}$}\label{alg:sigunsafe}
\STATE{$ G \gets G \cup \{ r \}$}
\ENDIF
\STATE{$\dotsc$}
\end{algorithmic}
\end{algorithm}

Note that $\deg(\spol r h )=d$ in Line~\ref{alg:sigunsafe} of
Algorithm~\ref{alg:ff2} is possible due to the restriction to sig-safe
reductions. $\spol r h$ then corresponds to a previously not handled, not
sig-safe reduction step of $r$.

\begin{rem}
In \cite{fF52002Corrected} a reduction with an element
of higher signature is solved in a slightly different way: 
Once noticed, the corresponding s-polynomial of higher
signature is generated. It is clear that for homogeneous input
this s-polynomial has the same degree as the other
elements already in $Q$. Thus it is directly added to $Q$, sorted in by
increasing signature. 
See \cite{epSig2011} for more information on this.
\end{rem}

Computing a Gr\"obner basis for 
an inhomogeneous ideal with $\ff$ the idea of homogenization can be 
used: One homogenizes the elements of $F$ w.r.t. some new variable,
call this $F^{\textrm{h}}$. Then a Gr\"obner basis $G^\textrm{h}$ for
$\langle F^\textrm{h} \rangle$ is computed w.r.t. a monomial ordering for which
the homogenization variable is smaller than all the other ones. 
Then one can receive a Gr\"obner 
basis $G$ for $\langle F \rangle$ from $G^\textrm{h}$.

This attempt has the advantage to compute step-by-step
intermediate Gr\"obner bases up to a given degree $d$, the degree of 
generated s-polynomials never drops.
Thus all possible reducers are available when they are needed.
In our ongoing discussion of $\ff$ this means that it is impossible 
that an element in $P$ will later on transform to a
new labeled polynomial in $G$ that could be useful for a reduction of an element
currently in $Q$.
On the other hand, the problem of this approach is that computing a Gr\"obner basis for 
$\langle F^\textrm{h} \rangle$ can be much harder than the computations for the
initial problem by adding solutions at infinity.

Efficiency of signature-based algorithms is based on handling critical pairs by increasing 
signatures. So in order to understand $\ff$'s restriction to
homogeneous input we need to answer the following two questions:
\begin{enumerate}
\item Does $\ff$ compute new elements by increasing signatures throughout the 
algorithm's working assuming homogeneous input?
\item If so, does this property get lost when applying $\ff$ to inhomogeneous
input?
\end{enumerate}
To answer these questions we need to find a connection between the degree of 
a labeled polynomial, that means, the degree of the polynomial part of it, and its 
signature.

Looking at the homogeneous situation first,
constructing s-polynomials has nice properties: Let $f$ and $g \in \lpset$
such that $\poly(f)$ as well as $\poly(g)$ are homogeneous. Computing corresponding 
multipliers $u$ and $v$ such that $u\lt(f) = v\lt(g)$ we can construct their 
s-polynomial $\spol f g$. Clearly, 
$u\poly(f)$ and $v\poly(g)$ are homogeneous, too.  It follows that 
\[\deg\left(\spol f g \right) = \deg\left(u\lt(f)\right) = \deg\left(v\lt(g)
    \right) = \deg(f,g).\]
For inhomogeneous $\poly(f)$ and $\poly(g)$, the situation is different as, for
example, $\deg\left(u\left(f-\lt(f)\right)\right)$ might be smaller than
$\deg\left( uf\right)$. So building the s-polynomial of $f$ and $g$ a drop for
the polynomial degree can happen:
\[\deg\left(\spol f g \right) \leq \deg\left(u\lt(f)\right) = \deg\left(v\lt(g)
    \right)  = \deg(f,g).\]
Next, let us see how $\ff$ handles the signatures and the coresponding degrees:
Input elements of $F$ are initialized to labeled polynomials 
$g_i = (e_i,f_i)$, by definition it holds that
\[\sigdeg\left( g_i \right) = \deg\left( g_i \right),\]
regardless of whether $\poly(g_i)$ is homogeneous or not.
Generating an s-polynomial of two elements $f$ and $g$ a drop in the degree
of the corresponding signature could only happen in the following situations:
\begin{enumerate}
\item $\sig(uf) = \sig(vg)$, which would mean that $\spol f g$ is non-minimal,
  and thus it would not be computed in $\ff$.
\item Once $\spol f g$ is built, the signature drops due to some ongoing
reduction of the polynomial part. This would not be a sig-safe reduction at this
degree step as well as at any upcoming higher degree. Thus such a reduction is not
processed.
\end{enumerate}
It follows that
\[\sigdeg(f,g) = \sigdeg\left(\spol f g\right) \geq \deg\left(\spol f
    g\right).\]
Thus new labeled polynomials $h$ are added to $G$ for which $\sigdeg(h) \geq
\deg(h)$ holds. Generating new critical pairs and s-polynomials from
this point on we come to the following relation for arbitrary $f$ and $g$ in $G$
computed by $\ff$ (or any other signature-based algorithm related to
    Algorithm~\ref{alg:sba}):
\begin{equation}
\label{eq:relation}
\sigdeg(f,g) = \sigdeg\left(\spol f g\right) \geq \deg(f,g) \geq \deg\left(\spol f
    g\right).
\end{equation}
Assuming homogeneous input to the algorithm Relation~\ref{eq:relation} becomes an equation.

So we conclude this discussion with the following facts:
\begin{enumerate}
\item In a signature-based Gr\"obner basis algorithm with homogeneous input the
degree of the critical pair (respectively the corresponding s-poylnomial) and
its signature degree coincide.
Therefore it is useless to presort the pair set $P$ by increasing degrees of 
the s-polynomials and later on sort $Q$ by increasing signatures: 
The signature of an s-polynomial in $Q$ is always smaller than the signature of an
element in $P$.
\item In the inhomogeneous situation the equality  between the degree and the
signature degree of a labeled polynomial need not hold any longer. Thus
presorting critical pairs by the polynomial degree can have bad influence on the
signature-based algorithms' inner working: A processing of the elements by
increasing signature can no longer be guaranteed.

Think about the following quite likely situation:\footnote{
For example, in {\tt Eco-11} such a situation happens 
hundreds of times.}
Let $(f,g)$ and $(f',g')$ be two critical pairs in $P$, and let $u,v$
respectively $u',v' \in \rp$ be the multipliers for 
$\spol f g$ respectively $\spol {f'}{g'}$. Assume that $u\sig(f) 
\siggt v\sig(g)$ and $u'\sig(f')\siggt v'\sig(g')$.  Moreover, 
assume that $\deg\left( uf \right) < \deg\left( u'f' \right)$. In $\ff$,
once all critical pairs of degree smaller than $\deg\left( uf \right)$ 
have been processed, $(f,g)$ is added to $Q$, whereas $(f',g')$ stays 
in $P$ and its further computation is postponed to a later point. In the 
situation of inhomogeneous polynomial data it is possible 
that $u\sig(f) \siggt u'\sig(f')$ as we just have seen.
Thus an element of higher signature is computed {\it before} an 
element of lower signature. 

\end{enumerate}

The main problem is that efficiency of signature-based algorithms which use
variants of the non-minimal signature criterion and the
rewritable signature criterion together with sig-safe reductions are based on this fact. 
So $\ff$'s presorting of critical pairs by their polynomial degree can lead to
way less efficient computations if the input data is inhomogeneous.

Moreover, we have seen that for homogeneous input $\ff$'s presorting
of critical pairs is useless and does not change anything w.r.t. the 
order in which the algorithm handles its critical pairs: Those are still processed by
increasing signatures. 
On the other hand, exactly this presorting interfers $\ff$
once it comes to inhomogeneous input data. 

Thus signature-based Gr\"obner basis algorithms should always be implemented without a
purely polynomial degree preselection in order to achieve a better efficiency.

\begin{convention}
In the following we can assume $\ff$ without polynomial degree
preselection, hence as a variant of Algorithm~\ref{alg:sba}.
\end{convention}

\begin{rem}
Due to the equality between the polynomial degree and the signature degree
of a labeled polynomial in the homogeneous situation, signature-based algorithms
are designed to handle Gr\"obner basis computations very well in this setting:
Discarding efficiently useless critical pairs and sorting them by polynomial degree is, in 
general, the best possible selection strategy in this case.
\end{rem}

\section{The connection between the signature degree and the sugar degree}
\label{sec:sugar}
In the last section we have seen that in the homogeneous
case signature-based algorithms handle critical pairs in an optimal order. 
Naturally, the question about the algorithms' usefulness in the inhomogeneous comes 
to one's mind.

In \cite{gmnrtSugarCube1991} Giovini, Mora, Niesi, Robbiano, and Traverso
introduce the notion of the {\em sugar degree} of a polynomial.  Later 
on, Bigatti, Caboara, and Robbiano describe the idea of a self-saturating 
variant of Buchberger's algorithm in \cite{bcr2011}. There, an in-depth 
discussion on the theoretical background of the sugar degree is given.
The idea behind this kind of degree is to improve a Gr\"obner basis 
computation for inhomogeneous input by giving it the flavour of a homogeneous 
one. Note that there exist other concepts for optimizations in the 
inhomogeneous setting, see for example~\cite{ufnarovski2008} or the idea of 
self-saturation given in \cite{bcr2011}\footnote{See also 
Section~\ref{sec:conclusion}.}. Still, the approach using the sugar degree is so far 
the most popular one due to its simple implementation. 

\begin{defi}
Let $F$ be the input of a purely polynomial Gr\"obner basis algorithm computing
$G$, let $f$ and $g$ be two elements in $G$, and let $t \in \monset$.
The {\bf sugar degree} is defined in the following way:
\begin{enumerate}
\item $\sdeg(f_i) := \deg(f_i)$ for all $i \in \{1,\dotsc,m\}$,
\item $\sdeg(tf) := \deg(t) + \sdeg(f)$, and 
\item $\sdeg(f+g) := \max\left\{\sdeg(f),\sdeg(g)\right\}$.
\end{enumerate}
For a critical pair $(f,g)$ we define the sugar degree by the sugar degree
of the corresponding s-polynomial, $\sdeg(f,g) := \sdeg\left( \spol f g \right)$.
\end{defi}

Clearly, if $F$ consists of homogeneous polynomials the sugar degree
coincides with degree throughout the Gr\"obner basis computation. When computing
with inhomogeneous data, the sugar degree becomes a useful tool: It  mimics the 
degree the elements would have, if the input sequence would have
been homogenized before starting the computations. Thus using the sugar degree the following
threepartite sorting of critical pairs emerges to be very efficient in a wide class 
of example sets tested (see \cite{gmnrtSugarCube1991} for more information on this):

\begin{enumerate}
\item increasing sugar degree,
\item increasing degree,
\item increasing w.r.t. $<$.
\end{enumerate}

Critical pairs are then sorted as in the homogeneous situation without the
overhead of homogenizing at all. Also this sorting needs not to be
optimal, it has a positive influence on the efficiency of Gr\"obner
basis algorithms in general.

Next we discuss how sorting critical pairs by increasing signatures is related 
to the ``sugared'' ordering.

\begin{thm}
\label{thm:sigdegsugardeg}
Let $f$ be a labeled polynomial appearing during a 
signature-based Gr\"obner basis computation, then $\sigdeg(f) = \sdeg\left( \poly(f) \right)$.
\end{thm}

\begin{proof}
For each $f_i \in F$ it holds that the initial labeled polynomial $g_i= 
(e_i,f_i)$ fulfills that $\sigdeg(g_i) = \deg(g_i) = \deg(f_i) = \sdeg(f_i)$.

Let $f$ and $g$ be two labeled polynomials in a signature-based Gr\"obner
basis algorithm. For $t \in \monset$ it holds that $\sigdeg(tf) = \deg(t) + 
\sigdeg(f)$.

Let $u$ and $v$ be multipliers in $\rp$ such that $u\lt(f) = v\lt(g)$.
Note that we can 
assume $\spol f g$ to be not non-minimal. W.l.o.g. let $\sig(uf)
\siggt \sig(vg)$. Then it holds that 
$\sigdeg\left( f,g \right) = \sigdeg\left( \spol f g \right) = \sigdeg(uf)$.

Thus the signature degree of a labeled polynomial $f$ in a signature-based Gr\"obner
basis algorithm coincides with the sugar degree of the corresponding polynomial
part $\poly(f)$.
\end{proof}

At this point of our discussion we need to distinguish possible choices for
$\prec$ on the signatures. Let us have a closer look at different 
situations assuming a degree compatible monomial ordering:
\begin{enumerate}
\item Using $\sch$ on the signatures, results in a non-incremental 
signature-based algorithm choosing critical pairs by increasing sugar degree.  
This is based on the fact that
\[ \sig(f) \sch \sig(g) \Longleftrightarrow \sigdeg(f) < \sigdeg(g).\]
\item Choosing $\pot$ the situation gets more complicated:
The algorithm prefers the signatures of higher module position. Assuming a 
critical pair being generated by two elements with signatures of different 
module positions it is not clear that the signature of higher module position
also has a higher degree. One can see that ordering by increasing sugar degree
in such a setting even breaks the incremental structure of the algorithm,
for example in {\tt Cyclic-5}.
\end{enumerate}
If a not degree compatible monomial ordering $<$ is given
then sorting critical pairs by increasing signature need not lead to a pair set 
sorted by increasing sugar degree. Even if we use $\sch$ on the signatures we 
cannot guarantee such a behaviour.

As yet, relaxing the restriction of sorting critical pairs 
by increasing signatures does not make
any sense. For example, ordering critical pairs by sugar degree signature-based
computations slow down by a large factor. 
The strengths of signature-based criteria detecting useless critical pairs and
sig-safe reductions are based on ordering by increasing signatures.
Disrespecting this fact a less efficient algorithm results. 
The signature ordering  $\siglt$ has to be preferred towards the polynomial
ordering $<$.

Nevertheless we can conclude that signature-based Gr\"obner basis algorithms choose by 
default a good selection strategy if a degree compatible monomial ordering is given, 
whether or not the input data is homogeneous. It coincides with the
sugar degree strategy if $\sch$ is used on the signatures.

\section{Still, there is a sour taste left}
\label{sec:problems}
The discovery of the last section seems to be compatible with published 
experimental results, for example, see~\cite{epSig2011}: Sometimes signature-based 
algorithms have problems computing Gr\"obner bases of inhomogeneous ideals, for example
computing {\tt Eco-11} is 9 times slower than computing {\tt Eco-11-h} in our
implementation of $\ff$ w.r.t. the graded reverse-lexicographical ordering.
For {\tt Eco-11} switching from $\pot$ to $\sch$ can 
improve timings (see Table~\ref{tab:timings}), still such an approach does not
always work. 
As we have seen, the selection strategy is efficient in the given example.
Also the signature-based criteria detecting useless critical pairs work good in the 
inhomogeneous setting, discarding a lot more elements than a Gebauer--M\"oller 
implementation.
So the only situation where problems can occur is the reduction process.

Since we lose the connection $\deg(f) = \sigdeg(f)$
for all labeled polynomials $f$ computed in signature-based algorithms for 
inhomogeneous input, forcing the reduction to be sig-safe 
can have a bad impact on the algorithms' behaviour.

In the homogeneous case the signature of the multiplied reducer can only be greater due
to either its signature index (considering $\pot$) or lexicographic
considerations (depending on the underlying monomial ordering $<$). It always
holds that $\sigdeg(f) = \sigdeg(tg)$
for $tg$ being a reducer of $f$, possibly sig-safe. Assuming the input to be inhomogeneous 
it is even possible that $\sigdeg(tg) \succ \sigdeg(f)$. The number of not sig-safe
reductions could increase compared to the homogeneous setting. This again means 
that a bunch of new critical 
pairs is generated, decreasing the algorithms' efficiency.
Assume in the above setting that the reduction itself is not allowed 
as $t\sig(g) \succ \sig(f)$. So in the following 
a new critical pair $(g,f)$ with signature $t\sig(g)$ is generated and later on 
computed. The problem is the ``later on'': Whereas $(g,f)$ has the same 
signature degree as $f$ in the homogeneous setting, assuming polynomials to be 
inhomogeneous it is possible that 
\[\sigdeg(g,f) = \sigdeg(tg) > \sigdeg(f).\]
This means that the corresponding data needed from the reduction step of $f$ and
$tg$ cannot be used in the algorithm at the time it really is needed. This triggers other 
reductions that would be helpful to take place at an earlier point of the 
algorithm to be delayed. 
Correctness is still ensured, the corresponding reduction steps needed for 
a Gr\"obner basis computation are executed nevertheless later on. As we are
assuming an underlying well-ordering on the polynomials the delay due to
introducing useless data has to be overruled after finitely many computational 
steps.\footnote{For more details on termination of signature-based algorithms we refer to
\cite{egpF52011}.}
Notwithstanding, the overhead that is computed due to these postponed reduction 
steps has a penalty on the performance of signature-based Gr\"obner basis 
algorithms.

We have tested a wide class of benchmarks from~\cite{posso}
and~\cite{graebeSymbolicData}, covering different admissible orderings,
finite and infinite ground fields, including parameters, etc. We have
implemented four different variants of signature-based Gr\"obner basis
algorithms ($\sba$) in the computer algebra system {\sc Singular}.
Since the implementations are still experimental, undergoing further development,
they are currently not part of the stable {\sc Singular} repository. Still, they
are publicly available in the branch {\tt sba}\footnote{In this paper we used
  commit {\tt 
  291021c19066befbcdd8a7d7626e5ddc2d421db4}.} at
\begin{center}
\url{https://github.com/ederc/Sources}.
\end{center}

In order not to overcharge the reader, we present in this 
paper results only for a part of our test suite. Those benchmarks represent the 
algorithms' overall behaviour quite accurately. The reader interested in the 
complete data set can get it at
\begin{center}
\url{https://github.com/ederc/benchmarks}.
\end{center}

The 4 variants of $\sba$ are tested for homogeneous and for inhomogeneous
input each, thus we must represent 8 different values per benchmark. We decided
to visualize our results by colorized bars.
See Table~\ref{tab:sba} for an overview of the implemented variants and the
color scheme chosen.

\begin{savenotes}
\begin{table}[H]
\begin{centering}
    \begin{tabular}{|c||c|c|}
		\hline
    \textbf{Alg $\backslash$ Ord} & $\pot$ & $\sch$\\
		\hline
    \hline
    $\ff$\footnote{presented in \cite{fF52002Corrected}, 
      including the optimizations mentioned in
      \cite{epF5C2009,epSig2011} and Section~\ref{sec:ff}}  & blue & orange\\
    \hline
    $\ap$\footnote{presented in \cite{apF5CritRevised2011},
      including the optimizations mentioned in \cite{epF5C2009,epSig2011}} & green & red\\
    \hline
  \end{tabular}
	\par
\end{centering}
\caption{Implemented signature-based Gr\"obner basis algorithms and color scheme
for Figure~\ref{fig1}}
\label{tab:sba}
\end{table}
\end{savenotes}

Warm colors represent computations done in a non-incremental way, cold ones
stand for incremental variants. In Figure~\ref{fig1} the darker 
variation illustrates results for the
respective inhomogeneous example, whereas the lighter variation stands for
results achieved computing the corresponding homogenized example.


In Figure~\ref{fig1} we give an overview of the behaviour of the 4 variants of
$\sba$ in several different benchmark sets, both homogeneous and inhomogeneous.
We lay our focus on the differences in the reduction process with a
look at the ratio between the number of higher signature detections and the
number of reduction steps in total. We give these ratios in percentage,
represented by the height of the respective bars in the diagrams. With this we 
would
like to get a better feeling for the influence of losing the connection between
$\deg(f)$ and $\sigdeg(f)$ in the inhomogeneous setting.
For an even better estimate we combine the ratios of 
Figure~\ref{fig1} with the timings given in Table~\ref{tab:timings}.

All examples where computed on an 
INTEL\textregistered~XEON\textregistered~X5460~@~3.16GHz processor with 64 GB of 
RAM and 120 GB of swap space running a 
  2.6.31--gentoo--r6 GNU/Linux 64--bit operating system.

Note that all of the examples presented in this paper are 
computed w.r.t. the graded reverse-lexicographical ordering. The complete 
benchmark set available online also includes computations 
w.r.t.  lexicographical orderings. Due to our discussion in 
Section~\ref{sec:sugar} signature-based computations w.r.t. not degree
compatible monomial orderings are rather inefficient in terms of sorting 
critical pairs. In such cases we found that it is more efficient to compute
a Gr\"obner basis w.r.t. the graded 
reverse-lexicographical ordering  and then to use a Gr\"obner 
conversion via $\fglm$ or even a Gr\"obner walk.

\begin{table}
\begin{centering}
    \begin{tabular}{|c|D{.}{.}{3}|D{.}{.}{3}|D{.}{.}{3}|D{.}{.}{3}|}
		\hline
      Test case & \multicolumn{1}{c|}{$\ff$,$\pot$} & 
      \multicolumn{1}{c|}{$\ap$,$\pot$} & \multicolumn{1}{c|}{$\ff$,$\sch$} & 
      \multicolumn{1}{c|}{$\ap$,$\sch$}\\   \hline
		\hline
    {\tt Cyclic-7} & 1.330 & 1.260 & 1.840 & 2.660 \\ \hline
    {\tt Cyclic-7-h} & 1.180 & 1.140 & 1.820 & 2.630\\ \hline
    {\tt Cyclic-8} & 468.260 & 442.870 & 314.970 & 184.900\\ \hline
    {\tt Cyclic-8-h} & 387.890 & 382.230 & 307.000 & 186.780\\ \hline
    {\tt Ext-Cyclic-6} & 157.590 & 129.340 & 13.420 & 16.770\\ \hline
    {\tt Ext-Cyclic-6-h} & 662.380 & 569.880 & 10.260 & 14.090\\ \hline
    {\tt Ilias-12} & 3,447.510 & 458.480 & 639.870 & 283.960\\ \hline
    {\tt Ilias-12-h} & 4,381.080 & 2,240.890 & 553.180 & 239.490\\ \hline
    {\tt Eco-10} & 45.610 & 2.780 & 7.990 & 7.190\\ \hline
    {\tt Eco-10-h} &14.990 & 13.280 & 3.660 & 4.290\\ \hline
    {\tt Eco-11} & 2,398.970 & 29.810 & 163.830 & 125.340 \\ \hline
    {\tt Eco-11-h} & 372.090 & 319.710 & 48.710 & 56.680\\ \hline
    {\tt Red-Eco-11} & 2.620 & 2.430 & 14.530 & 14.580\\ \hline
    {\tt Red-Eco-11-h} & 2.600 & 2.460 & 19.290 & 18.760\\ \hline
    {\tt Red-Eco-12} & 22.010 & 20.330 & 161.530 & 158.430\\ \hline
    {\tt Red-Eco-12-h} & 21.390 & 18.610 & 246.470 & 241.590\\ \hline
    {\tt F-744} & 1.550 & 0.740 & 0.430 & 0.450 \\ \hline
    {\tt F-744-h} & 2.090 & 1.470 & 0.360 & 0.380\\ \hline
    {\tt F-855} & 50.670 & 27.200 & 122.670 & 96.080\\ \hline
    {\tt F-855-h} & 133.470 & 65.980 & 48.600 & 48.930\\ \hline
    {\tt Fabrice-24} & 101.900 & 72.250 & 113.710 & 108.300\\ \hline
    {\tt Fabrice-24-h} & 121.900 & 101.190 & 361.570 & 326.040\\ \hline
    {\tt Katsura-12} & 111.690 & 61.490 & 1,287.250 & 1,303.360\\ \hline
    {\tt Katsura-12-h} & 109.970 & 54.510 & 1,260.380 & 1,223.710\\ \hline
  \end{tabular}
	\par
\end{centering}
\caption{Timings in seconds for the computation of a Gr\"obner basis for the 
  given test case.}
\label{tab:timings}
\end{table}

The results in Figure~\ref{fig1} are rather ambiguous:
 Sometimes the ratio is
several times greater in the inhomogeneous setting than in the corresponding
homogeneous one
(see, for example, {\tt Cyclic-8} and {\tt Ext-Cyclic-6} for $\ff$ and $\ap$
 using $\pot$). Whereas in examples like {\tt Ilias-12} it is just the
other way around.

In various examples the number of sig-safe reduction steps is
a factor of $1000$ greater than the number of higher signature
detections, for example, see {\tt Noon-n} or {\tt Katsura-n}. Not
depending on whether the input is homogeneous or not, the influence of not
sig-safe data is not even measureable in these cases.

Talking about incremental versus non-incremental computations there is an
inclination that the ratio of the number of higher signature detections and the
number of reduction steps done is mostly smaller in the non-incremental setting.
Still one
needs to keep in mind that Figure~\ref{fig1} presents only the ratios: 
For example, in {\tt Katsura-12} the non-incremental variants of $\sba$
are multiple times slower than the incremental ones (see
    Table~\ref{tab:timings}), they do approximately 50 times more reduction
steps. 
Due to this high amount of reductions the ratio gets
lower. Furthermore, finding a heuristic when to prefer incremental computations over
non-incremental ones, for example, see {\tt Katsura-n}, is of great importance.

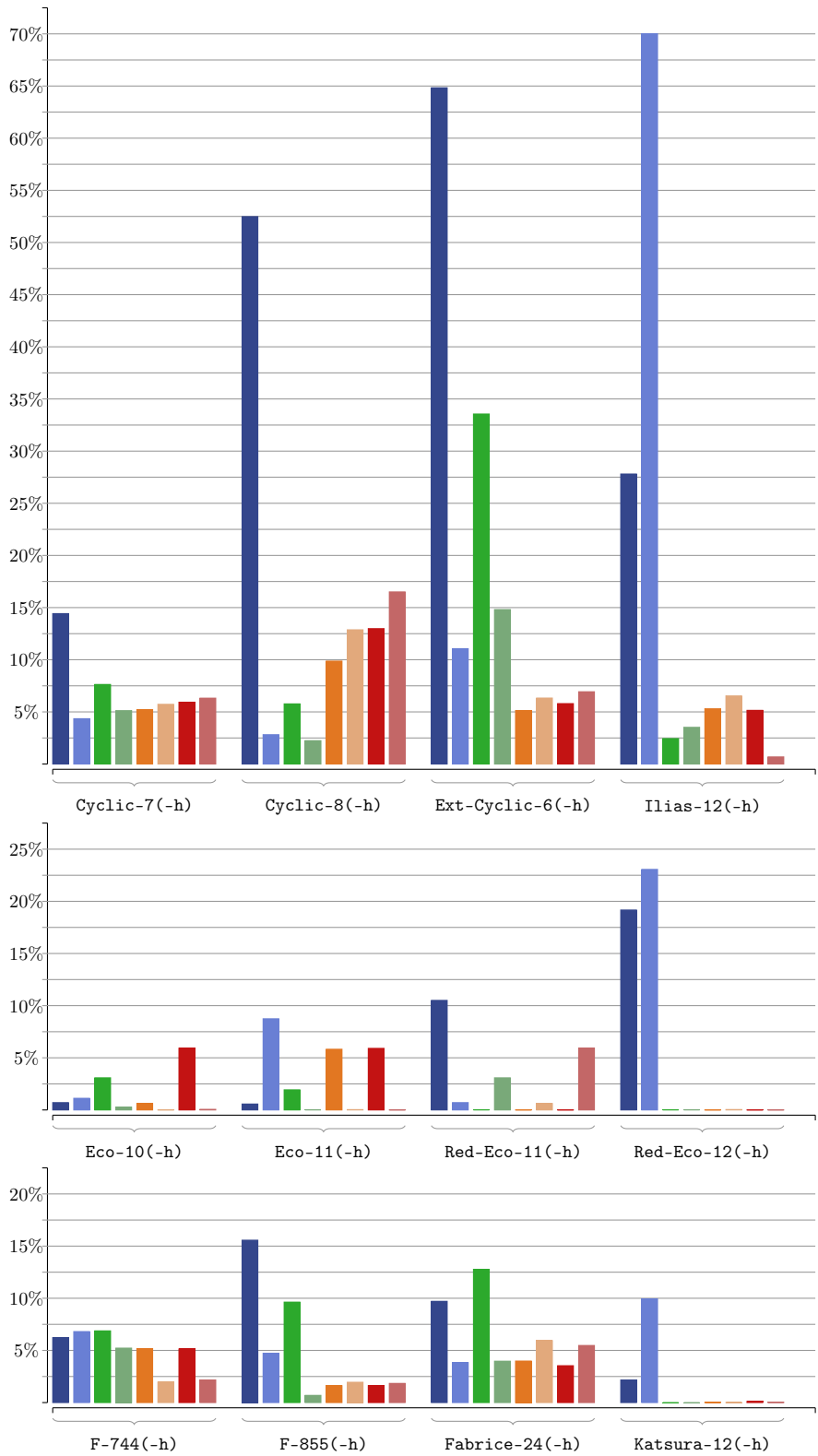
\begin{figure}
\begin{tikzpicture}[scale=.75, transform shape]  
  \draw (0cm,-0.1cm) -- (14.5cm,-0.1cm);  
  \draw (0cm,-0.1cm) -- (0cm,-0.2cm);  
  \draw (14.5cm,-0.1cm) -- (14.5cm,-0.2cm);  
  
  \draw (-0.1cm,0cm) -- (-0.1cm,14.5cm);  
  \draw (-0.1cm,0cm) -- (-0.2cm,0cm);  
  \draw (-0.1cm,14.5cm) -- (-0.2cm,14.5cm);  

  \draw[gray!80, text=black] 
    node at (-0.5 cm,1cm) {$5\%$}  
    node at (-0.5 cm,2cm) {$10\%$}  
    node at (-0.5 cm,3cm) {$15\%$}  
    node at (-0.5 cm,4cm) {$20\%$}  
    node at (-0.5 cm,5cm) {$25\%$}  
    node at (-0.5 cm,6cm) {$30\%$}  
    node at (-0.5 cm,7cm) {$35\%$}  
    node at (-0.5 cm,8cm) {$40\%$}  
    node at (-0.5 cm,9cm) {$45\%$}  
    node at (-0.5 cm,10cm) {$50\%$}  
    node at (-0.5 cm,11cm) {$55\%$}  
    node at (-0.5 cm,12cm) {$60\%$}  
    node at (-0.5 cm,13cm) {$65\%$}  
    node at (-0.5 cm,14cm) {$70\%$};  
  \foreach \a in {1,...,14}  
    \draw[gray!80, text=black] (-0.2 cm,\a cm) -- (14.5 cm,\a cm); \foreach \a 
    in {0.5,1.5,...,13.5}  
    \draw[gray!80, text=black] (-0.2 cm,\a cm) -- (14.5 cm,\a cm); 

  \foreach
  \a/\b\rff/\hff/\tff/\rap/\hap/\tap/\rnff/\hnff/\tnff/\rnap/\hnap/\tnap/\bench in 
    {
      0/0/3.14340/0.45343/1.330/2.20818/0.16870/1.260/4.15226/0.21687/1.840/3.33459/0.19773/2.660/{\tt Cyclic-7},
      0/3.6/15.281111/8.019624/468.260/5.709711/0.329421/540.860/11.988665/1.183386/314.970/6.086544/0.790262/184.900/{\tt
        Cyclic-8},
      0/7.2/10.293657/6.673612/1.330/22.92181/7.69232/1.260/16.46813/0.84469/1.840/13.17033/0.76312/2.660/{\tt
        Ext-Cyclic-6},
      0/10.8/35.972206/10.00305/111.690/78.84253/1.92455/61.490/90.00558/4.77048/1287.250/
        46.45496/2.39221/1303.360/{\tt Ilias-12}
    }
    {
     \draw[color=myblued, fill=myblued] (\b cm,\a cm+0.00cm) rectangle (\b cm +
         0.3 cm,\hff cm / \rff cm * 20 cm + \a cm+0.00cm); 

     \draw[color=mygreend, fill=mygreend] (0.8cm + \b cm,\a cm+0.00cm) rectangle
     (\b cm + 0.8 cm + 0.3cm,\hap cm / \rap cm * 20cm +\a cm+0.00cm); 

     \draw[color=myyellowd, fill=myyellowd] (1.6cm + \b cm,\a cm+0.00cm) rectangle
     (\b cm + 1.6 cm + 0.3cm,\hnff cm / \rnff cm * 20cm +\a cm+0.00cm); 

     \draw[color=myredd, fill=myredd] (2.4cm + \b cm,\a cm+0.00cm) rectangle
     (\b cm + 2.4 cm + 0.3cm,\hnap cm / \rnap cm * 20cm+\a cm+0.00cm); 

    };
  \foreach
  \a/\b\rff/\hff/\tff/\rap/\hap/\tap/\rnff/\hnff/\tnff/\rnap/\hnap/\tnap/\bench in 
    {
      0/0.4/2.85592/0.12418/1.180/2.00889/0.10293/1.140/4.18583/0.23974/1.820/3.35769/0.21218/2.630/{\tt
        Cyclic-7-h},
      0/4/13.151219/0.370661/387.890/4.922021/0.110325/508.190/12.193809/1.569424/307.230/6.253714/1.032245/186.780/{\tt
        Cyclic-8-h},
      0/7.6/15.582246/1.723938/1.330/30.43281/4.50702/1.260/16.55864/1.04793/1.840/13.23665/0.91741/2.660/{\tt
        Ext-Cyclic-6-h},
      0/11.2/11.2496010/7.876296/111.690/35.716155/1.258523/61.490/89.53091/5.8511/1287.250/
        46.50775/0.32033/1303.360/{\tt Ilias-12-h}
    }
    {
     \draw[color=mybluel, fill=mybluel] (\b cm,\a cm+0.00cm) rectangle (\b cm +
         0.3 cm,\hff cm / \rff cm * 20 cm + \a cm+0.00cm); 

     \draw[color=mygreenl, fill=mygreenl] (0.8cm + \b cm,\a cm+0.00cm) rectangle
     (\b cm + 0.8 cm + 0.3cm,\hap cm / \rap cm * 20cm +\a cm+0.00cm); 

     \draw[color=myyellowl, fill=myyellowl] (1.6cm + \b cm,\a cm+0.00cm) rectangle
     (\b cm + 1.6 cm + 0.3cm,\hnff cm / \rnff cm * 20cm +\a cm+0.00cm); 

     \draw[color=myredl, fill=myredl] (2.4cm + \b cm,\a cm+0.00cm) rectangle
     (\b cm + 2.4 cm + 0.3cm,\hnap cm / \rnap cm * 20cm+\a cm+0.00cm); 

    };
    \draw[decorate,decoration={brace,mirror},gray!80] (0cm,-0.3cm) -- (3.1cm,-0.3cm);
    \node[] at (1.55cm,-0.8cm) { {\tt Cyclic-7(-h)} };  
    \draw[decorate,decoration={brace,mirror},gray!80] (3.6cm,-0.3cm) -- (6.7cm,-0.3cm);
    \node[] at (5.15cm,-0.8cm) { {\tt Cyclic-8(-h)} };  
    \draw[decorate,decoration={brace,mirror},gray!80] (7.2cm,-0.3cm) -- (10.3cm,-0.3cm);
    \node[] at (8.75cm,-0.8cm) { {\tt Ext-Cyclic-6(-h)} };  
    \draw[decorate,decoration={brace,mirror},gray!80] (10.8cm,-0.3cm) -- (13.9cm,-0.3cm);
    \node[] at (12.35cm,-0.8cm) { {\tt Ilias-12(-h)} };  

\end{tikzpicture}

\begin{tikzpicture}[scale=.75, transform shape]  
  \draw (0cm,-0.1cm) -- (14.5cm,-0.1cm);  
  \draw (0cm,-0.1cm) -- (0cm,-0.2cm);  
  \draw (14.5cm,-0.1cm) -- (14.5cm,-0.2cm);  
  
  \draw (-0.1cm,0cm) -- (-0.1cm,5.5cm);  
  \draw (-0.1cm,0cm) -- (-0.2cm,0cm);  
  \draw (-0.1cm,5.5cm) -- (-0.2cm,5.5cm);  

  \draw[gray!80, text=black] 
    node at (-0.5 cm,1cm) {$5\%$}  
    node at (-0.5 cm,2cm) {$10\%$}  
    node at (-0.5 cm,3cm) {$15\%$}  
    node at (-0.5 cm,4cm) {$20\%$}  
    node at (-0.5 cm,5cm) {$25\%$};  
  \foreach \a in {1,...,5}  
    \draw[gray!80, text=black] (-0.2 cm,\a cm) -- (14.5 cm,\a cm); 
  \foreach \a in {0.5,1.5,...,4.5}  
    \draw[gray!80, text=black] (-0.2 cm,\a cm) -- (14.5 cm,\a cm); 

  \foreach
  \a/\b\rff/\hff/\tff/\rap/\hap/\tap/\rnff/\hnff/\tnff/\rnap/\hnap/\tnap/\bench in 
    {
      0/0/19.07901/0.13444/45.610/2.16420/0.06671/2.780/6.59802/0.040842/7.990/6.02291/0.35696/7.190/{\tt
        Eco-10},
      0/3.6/187.3065/1.06614/2398.970/11.82142/0.22825/29.810/35.53883/2.06904/0/31.49549/1.85992/125.340/{\tt
        Eco-11},
      0/7.2/16.686/1.753/2398.970/5.492/0/29.810/19.56198/0.00066/0/19.25930/0.00066/125.340/{\tt
        Red-Eco-11},
      0/10.8/10.2764/1.9700/45.610/16.334/0/2.780/9.234571/0.000083/7.990/9.140222/0.000083/7.190/{\tt
        Red-Eco-12}
    }
    {
     \draw[color=myblued, fill=myblued] (\b cm,\a cm+0.00cm) rectangle (\b cm +
         0.3 cm,\hff cm / \rff cm * 20 cm + \a cm+0.00cm); 

     \draw[color=mygreend, fill=mygreend] (0.8cm + \b cm,\a cm+0.00cm) rectangle
     (\b cm + 0.8 cm + 0.3cm,\hap cm / \rap cm * 20cm +\a cm+0.00cm); 

     \draw[color=myyellowd, fill=myyellowd] (1.6cm + \b cm,\a cm+0.00cm) rectangle
     (\b cm + 1.6 cm + 0.3cm,\hnff cm / \rnff cm * 20cm +\a cm+0.00cm); 

     \draw[color=myredd, fill=myredd] (2.4cm + \b cm,\a cm+0.00cm) rectangle
     (\b cm + 2.4 cm + 0.3cm,\hnap cm / \rnap cm * 20cm+\a cm+0.00cm); 

    };
  \foreach
  \a/\b\rff/\hff/\tff/\rap/\hap/\tap/\rnff/\hnff/\tnff/\rnap/\hnap/\tnap/\bench in 
    {
      0/0.4/4.55339/0.05094/14.990/1.85506/0.004997/13.280/5.88594/0.000467/3.660/5.66915/0.00341/4.290/{\tt
        Eco-10-h},
      0/4/17.194/1.503/2398.970/5.492/0/29.810/21.79419/0.00686/206904163.830/19.92199/0.00066/125.340/{\tt
        Eco-11-h},
      0/7.6/1.907901/0.013444/45.610/2.16420/0.06671/2.780/6.59802/0.040842/7.990/6.02291/0.35696/7.190/{\tt
        Red-Eco-11-h},
      0/11.2/7.6739/1.7698/45.610/1.6334/0/2.780/10.663441/0.004157/7.990/9.472116/0.000083/7.190/{\tt
        Red-Eco-12-h}
    }
    {
     \draw[color=mybluel, fill=mybluel] (\b cm,\a cm+0.00cm) rectangle (\b cm +
         0.3 cm,\hff cm / \rff cm * 20 cm + \a cm+0.00cm); 

     \draw[color=mygreenl, fill=mygreenl] (0.8cm + \b cm,\a cm+0.00cm) rectangle
     (\b cm + 0.8 cm + 0.3cm,\hap cm / \rap cm * 20cm +\a cm+0.00cm); 

     \draw[color=myyellowl, fill=myyellowl] (1.6cm + \b cm,\a cm+0.00cm) rectangle
     (\b cm + 1.6 cm + 0.3cm,\hnff cm / \rnff cm * 20cm +\a cm+0.00cm); 

     \draw[color=myredl, fill=myredl] (2.4cm + \b cm,\a cm+0.00cm) rectangle
     (\b cm + 2.4 cm + 0.3cm,\hnap cm / \rnap cm * 20cm+\a cm+0.00cm); 

    };
    \draw[decorate,decoration={brace,mirror},gray!80] (0cm,-0.3cm) -- (3.1cm,-0.3cm);
    \node[] at (1.55cm,-0.8cm) { {\tt Eco-10(-h)} };  
    \draw[decorate,decoration={brace,mirror},gray!80] (3.6cm,-0.3cm) -- (6.7cm,-0.3cm);
    \node[] at (5.15cm,-0.8cm) { {\tt Eco-11(-h)} };  
    \draw[decorate,decoration={brace,mirror},gray!80] (7.2cm,-0.3cm) -- (10.3cm,-0.3cm);
    \node[] at (8.75cm,-0.8cm) { {\tt Red-Eco-11(-h)} };  
    \draw[decorate,decoration={brace,mirror},gray!80] (10.8cm,-0.3cm) -- (13.9cm,-0.3cm);
    \node[] at (12.35cm,-0.8cm) { {\tt Red-Eco-12(-h)} };  

\end{tikzpicture}

\begin{tikzpicture}[scale=.75, transform shape]  
  \draw (0cm,-0.1cm) -- (14.5cm,-0.1cm);  
  \draw (0cm,-0.1cm) -- (0cm,-0.2cm);  
  \draw (14.5cm,-0.1cm) -- (14.5cm,-0.2cm);  
  
  \draw (-0.1cm,0cm) -- (-0.1cm,4.5cm);  
  \draw (-0.1cm,0cm) -- (-0.2cm,0cm);  
  \draw (-0.1cm,4.5cm) -- (-0.2cm,4.5cm);  

  \draw[gray!80, text=black] 
    node at (-0.5 cm,1cm) {$5\%$}  
    node at (-0.5 cm,2cm) {$10\%$}  
    node at (-0.5 cm,3cm) {$15\%$}  
    node at (-0.5 cm,4cm) {$20\%$};  
  \foreach \a in {1,...,4}  
    \draw[gray!80, text=black] (-0.2 cm,\a cm) -- (14.5 cm,\a cm); 
  \foreach \a in {0.5,1.5,...,3.5}  
    \draw[gray!80, text=black] (-0.2 cm,\a cm) -- (14.5 cm,\a cm); 

  \foreach
  \a/\b\rff/\hff/\tff/\rap/\hap/\tap/\rnff/\hnff/\tnff/\rnap/\hnap/\tnap/\bench in 
    {
      0/0/3.32464/0.20695/1.550/1.11161/0.07642/0.740/1.69091/0.08754/0.430/1.54060/0.07918/0.450/{\tt F-744},
      0/3.6/30.97114/4.82161/50.670/8.79282/.84669/27.200/79.27247/1.28238/122.670/59.23830/.96104/96.080/{\tt F-855},
      0/7.2/32.912/3.195/1.550/14.951/1.910/0.740/88.513/3.524/0.430/82.626/2.920/0.450/{\tt
        Fabrice-24},
      0/10.8/19.26302/0.41738/111.690/4.16945/0/61.490/24.545763/0.004588/1287.250/
       23.43906 /0.02696/1303.360/{\tt Katsura-12}
    }
    {
     \draw[color=myblued, fill=myblued] (\b cm,\a cm+0.00cm) rectangle (\b cm +
         0.3 cm,\hff cm / \rff cm * 20 cm + \a cm+0.00cm); 

     \draw[color=mygreend, fill=mygreend] (0.8cm + \b cm,\a cm+0.00cm) rectangle
     (\b cm + 0.8 cm + 0.3cm,\hap cm / \rap cm * 20cm +\a cm+0.00cm); 

     \draw[color=myyellowd, fill=myyellowd] (1.6cm + \b cm,\a cm+0.00cm) rectangle
     (\b cm + 1.6 cm + 0.3cm,\hnff cm / \rnff cm * 20cm +\a cm+0.00cm); 

     \draw[color=myredd, fill=myredd] (2.4cm + \b cm,\a cm+0.00cm) rectangle
     (\b cm + 2.4 cm + 0.3cm,\hnap cm / \rnap cm * 20cm+\a cm+0.00cm); 

    };
  \foreach
  \a/\b\rff/\hff/\tff/\rap/\hap/\tap/\rnff/\hnff/\tnff/\rnap/\hnap/\tnap/\bench in 
    {
      0/0.4/5.95495/0.40529/1.550/3.86739/0.20193/0.740/1.50498/0.02954/0.430/1.40405/0.03039/0.450/{\tt
        F-744-h},
      0/4/99.60257/4.70587/1.550/34.29406/0.23711/0.740/53.43879/1.04124/0.430/42.98564/0.79019/0.450/{\tt
        F-855-h},
      0/7.6/111.405/4.265/1.550/66.967/2.658/0.740/142.792/8.477/0.430/134.157/7.342/0.450/{\tt
        Fabrice-24-h},
      0/11.2/1.802451/0.17921/111.690/3.90042/0/61.490/26.980103/0.0010431/1287.250/
       25.476540/0.008091/1303.360/{\tt Katsura-12}
    }
    {
     \draw[color=mybluel, fill=mybluel] (\b cm,\a cm+0.00cm) rectangle (\b cm +
         0.3 cm,\hff cm / \rff cm * 20 cm + \a cm+0.00cm); 

     \draw[color=mygreenl, fill=mygreenl] (0.8cm + \b cm,\a cm+0.00cm) rectangle
     (\b cm + 0.8 cm + 0.3cm,\hap cm / \rap cm * 20cm +\a cm+0.00cm); 

     \draw[color=myyellowl, fill=myyellowl] (1.6cm + \b cm,\a cm+0.00cm) rectangle
     (\b cm + 1.6 cm + 0.3cm,\hnff cm / \rnff cm * 20cm +\a cm+0.00cm); 

     \draw[color=myredl, fill=myredl] (2.4cm + \b cm,\a cm+0.00cm) rectangle
     (\b cm + 2.4 cm + 0.3cm,\hnap cm / \rnap cm * 20cm+\a cm+0.00cm); 

    };
    \draw[decorate,decoration={brace,mirror},gray!80] (0cm,-0.3cm) -- (3.1cm,-0.3cm);
    \node[] at (1.55cm,-0.8cm) { {\tt F-744(-h)} };  
    \draw[decorate,decoration={brace,mirror},gray!80] (3.6cm,-0.3cm) -- (6.7cm,-0.3cm);
    \node[] at (5.15cm,-0.8cm) { {\tt F-855(-h)} };  
    \draw[decorate,decoration={brace,mirror},gray!80] (7.2cm,-0.3cm) -- (10.3cm,-0.3cm);
    \node[] at (8.75cm,-0.8cm) { {\tt Fabrice-24(-h)} };  
    \draw[decorate,decoration={brace,mirror},gray!80] (10.8cm,-0.3cm) -- (13.9cm,-0.3cm);
    \node[] at (12.35cm,-0.8cm) { {\tt Katsura-12(-h)} };  

\end{tikzpicture}
\caption{Ratio of higher-signature detections to reduction steps for various
  benchmarks}
\label{fig1}
\end{figure}

\begin{rem}\
\begin{enumerate}
\item A discussion on the differences of the implementations of the non-minimal 
signature criterion and the rewritable signature criterion those $4$ algorithms 
use is not in the focus of this paper.  
We refer to the corresponding papers for more details. \cite{epSig2011}
and~\cite{erF5SB2013} give an
overview on how the $4$ variants are related to each other.
Note that combining $\ff$ with $\sch$ does not introduce any theoretical 
problems for correctness of the algorithm.

Note that in various low-level implementations in {\sc Singular} $\ggv$ respectively 
$\gvw$ were not competitive to the 4
signature-based algorithms presented here. The lack of a real implementation of
the rewritable signature criterion seems to be the reason for this, we refer to
\cite{epSig2011}. In~\cite{volnyGVW2011} $\gvwhs$ is presented, a variant of
$\gvw$ using the rewritable criterion of $\ap$. This algorithm as well as the
recently by Roune and Stillman in~\cite{rounestillman2012} presented $\rs$ algorithm
coincide with our $\ap$ implementation.

\item The number of higher signature detections also depends
on the order in which the list of possible reducers is searched through. 
We can state that in most benchmarks, again independent of the
homogeneity of the input polynomials, using the settings and heuristics of {\sc Singular}'s
internal, Gebauer-M\"oller-like Gr\"obner basis algorithm {\tt groebner} is 
a good choice. Of course there are examples like {\tt Fabrice-24} where
adjusting the search by hand leads to an improvement in timings of a factor of
$10$, but in other examples exactly this choice slows down
computations by a factor of $100$ and even more. Finding good heuristics for searching 
in the set of reducers is an open problem; doing this by
increasing respectively descreasing signature is not a good choice in a wide
range of example classes.
\end{enumerate}
\end{rem}

\section{Conclusion and further research}
\label{sec:conclusion}
We have given an in-depth discussion about the behaviour of
signature-based Gr\"obner basis algorithms in the inhomogeneous case. 

Explaining, why $\ff$, as initially presented in \cite{fF52002Corrected} is
restricted to homogeneous input data, we found a solution for relaxing this
condition. Moreover, by doing this the presentation of the algorithm simplifies. 
This makes it easier for a reader without prior knowledge of
signature-based algorithms to get access to this area of Gr\"obner basis
theory.

Furthermore, we have presented for the first time the strong connection between the
signature degree and the sugar degree of the corresponding 
polynomial parts. It is a delightful discovery that signature-based algorithms 
sort the corresponding pair set in a nearly optimal order from the polynomial 
point of view when assuming a degree compatible monomial ordering.
Reordering critical pairs is bounded by the 
condition of computing by increasing signatures. The question if we can find 
more efficient orderings on the signatures in these situations remains
unanswered and needs further investigation. 

Investigating the suspicion that the lost connection between polynomial degree and signature 
degree in the inhomogeneous setting can affect the sig-safe reduction process
negatively cannot be confirmed. There are specific examples where the number of
higher signature detections increase strongly in the inhomogeneous setting
(compared to the homogeneous one), but there are also
examples behaving just the other way around.

Further investigations might be done in the direction of combining
sig-safe reduction steps with the idea of self-saturation given in \cite{bcr2011}.
The overall idea of self-saturation is to use special 
kinds of reduction steps to achieve so-called {\it (weak) 
saturating remainders}. Thereby the Gr\"obner basis algorithm starts with the 
homogenized set of generators, but instead of plainly computing the 
homogeneous Gr\"obner basis of the homogenized input data, reducers 
respectively the remainders of reductions are exchanged by saturated pendants.  
The process of self-saturation has a  positive effect on Buchberger-like Gr\"obner 
basis algorithms as shown in \cite{bcr2011}. 
Being restricted to sig-safe reductions in signature-based algorithms the freedom 
of choice for the saturated elements is limited and might break its positive
effects on the computations.

%
%
\begin{acknow}
The author would like to thank the {\sc Singular} team at the University of
Kaiserslautern for their support.
Moreover, the author especially wishes to thank the anonymous referees whose comments improved the
paper.
\end{acknow}
\bibliographystyle{abbrv}

\end{document}